\documentclass[letterpaper, 10 pt, conference]{ieeeconf}  % Comment this line out if you need a4paper
\usepackage{graphicx}          % Include this line if your 
\usepackage{subcaption}
\usepackage{algpseudocode}
\usepackage{amsmath,amssymb}
\usepackage{amsfonts}
\usepackage{xcolor}
\usepackage{bm}
\usepackage{tikz}
\usepackage{algorithm}
\usepackage{algpseudocode}

\newcommand{\norm}[1]{\left\lVert#1\right\rVert}

\newtheorem{assumption}{Assumption}
\newtheorem{proposition}{Proposition}
\newtheorem{definition}{Definition}

\IEEEoverridecommandlockouts                              % This command is only needed if 
\title{\LARGE \bf
META-SMGO-$\Delta$: similarity as a prior in black-box optimization
}

\author{Riccardo Busetto$^{1}$, Valentina Breschi$^{2}$ and Simone Formentin$^{1}$% <-this % stops a space stops a space
\thanks{$^{1}$R. Busetto and S. Formentin are with Dip. di Elettronica, Informazione e Bioingegneria, Politecnico di Milano, Milano, Italy.
        {\tt\small {name}.{surname}@polimi.it}}%
\thanks{$^{2}$V. Breschi is with Dept. of Electrical Engineering, Eindhoven University of Technology, Eindhoven, Netherlands .
        {\tt\small v.breschi@tue.nl}}%
\thanks{This work was partially supported by the Italian Ministry of University and Research under the PRIN’17 project “Data-driven learning of constrained control systems”, contract no. 2017J89ARP.}
}

\begin{document}

\maketitle
\thispagestyle{empty}
\pagestyle{empty}

%%%%%%%%%%%%%%%%%%%%%%%%%%%%%%%%%%%%%%%%%%%%%%%%%%%%%%%%%%%%%%%%%%%%%%%%%%%%%%%%
\begin{abstract}
When solving global optimization problems in practice, one often ends up repeatedly solving problems that are similar to each others. By providing a rigorous definition of similarity, in this work we propose to incorporate the META-learning rationale into SMGO-$\Delta$, a global optimization approach recently proposed in the literature, to exploit priors obtained from similar past experience to efficiently solve new (similar) problems. Through a benchmark numerical example we show the practical benefits of our META-extension of the baseline algorithm, while providing theoretical bounds on its performance.
\end{abstract}
%%%%%%%%%%%%%%%%%%%%%%%%%%%%%%%%%%%%%%%%%%%%%%%%%%%%%%%%%%%%%%%%%%%%%%%%%%%%%%%%
\section{INTRODUCTION}
Black-box optimization is a fundamental tool in many areas of science and engineering, where the objective function and/or the constraints are unknown or expensive to evaluate. Such an approach to optimization has been studied extensively in the literature, with many different algorithms proposed for solving this problem. One popular approach is evolutionary algorithms, which are inspired by biological evolution and use techniques such as mutation and selection to search for optimal solutions \cite{back1996evolutionary}. Another approach is Bayesian optimization, which uses probabilistic models to guide the search process and has been shown to be effective in many applications \cite{brochu2010tutorial}. Other methods include gradient-based optimization \cite{nocedal2006numerical}, simulated annealing \cite{kirkpatrick1983optimization}, and particle swarm optimization \cite{kennedy1995particle}.

Despite the large number of algorithms available, black-box optimization remains a challenging problem, particularly in high-dimensional spaces or when the objective function is noisy or non-convex. Specifically, when the vector of parameters is large, one might require several iterations before converging to a satisfactory solution.
To avoid this, practitioners usually acquire expertise by repeatedly solving problems that, though different, share common features with each other, i.e., they are somehow \textit{similar}.
The idea behind this paper starts from the observation that the experience acquired in solving optimization problems with shared characteristics can be re-used as valuable prior to solve more efficiently a new, similar problem.
This intuition is also at the foundation of \emph{meta-learning}, a subfield of machine learning that focuses on developing algorithms that can automatically learn how to solve new tasks more efficiently and effectively by leveraging prior experience from similar tasks \cite{thrun1998learning, vilalta2002perspective, bengio2012practical}. Meta-learning algorithms typically learn a mapping from task descriptions to model parameters, which can then be adapted to new tasks with few or no examples \cite{finn2017model, ravi2018meta}.

There are several different approaches to meta-learning, including metric-based methods, model-based methods, and optimization-based methods \cite{vanschoren2018meta}. Metric-based methods learn a distance metric between tasks that can be used to transfer knowledge from one task to another. Model-based methods learn a shared model architecture that can be adapted to new tasks by adjusting its weights. Optimization-based methods learn to optimize the learning process itself, such as by learning to initialize the model parameters or by learning to select the hyperparameters that control the learning process \cite{feurer2015initializing,rothfuss2023meta}.

Meta-learning has a wide range of applications in machine learning, including few-shot learning, transfer learning, and adaptive optimization \cite{wang2020generalizing, finn2017model, li2019episodic}. As such, it is an active area of research with many exciting new developments and open research questions. However, apart from few, very recent, exceptions (see, e.g., \cite{arcari2020meta,richards2022control,guo2023imitation,ecker2023data}), it has seldom been exploited in optimization and learning for control. Indeed, several control algorithms, especially those based on the use of experimental data, require definition of hyper-parameters that significantly affect the resulting closed-loop performance \cite{chakrabarty2022optimizing}. In general, these hyper-parameters are determined with expensive experiments, involving either sensitivity analyses, or driven by expert-based and rule-of-thumb-based design.
In this work, we take a recently introduced black-box optimization technique called SMGO-$\Delta$ \cite{sabug2022smgo}, based on set-membership identification and with the advantage of being simple and effective as compared to similar approaches like Bayesian optimization, and we apply meta-learning tools to show that prior experience with similar problems can be used to boost SMGO-$\Delta$ performance. In particular, we will provide theoretical bound for the performance, and we will show on a numerical case study that the resulting META-version of SMGO-$\Delta$ outperforms its standard implementation in terms of number of iterations required to converge and number of constraint violations.

The remainder of the paper is organized as follows. In Section~\ref{sec:background}, we review the main features of SMGO-$\Delta$, which are of paramount importance for the self-consistency of the work. These preliminaries allows us to formulate the problem stated in Section~\ref{sec:problem} and to discuss the proposed META extension of SMGO-$\Delta$, introduced in Section~\ref{sec:meta}. In this section we further prove two key properties of META-SMGO-$\Delta$, highlighting its potential advantages over the baseline algorithm, and we show how the proposed approach can be implemented in practice. The effectiveness of META-SMGO-$\Delta$ is then discussed in Section~\ref{sec:experiments}, showing the possible advantages of exploiting similarities for hyper-parameter tuning. The paper is ended by some concluding remarks.

%\subsection{Notation}
%\begin{itemize}
%    \item \textbf{Bold} symbols indicate vectors and matrices containing history of the optimization process, i.e., the number of columns is equivalent to the number of iterations $i=1,\ldots,n$.
%    E.g., $\bm{X}=[X^{(1)},\ldots,X^{(n)}]$, $\bm{z}=[z^{(1)},\ldots,z^{(n)}]$ are the matrix and the row-vector containing history of sampled points during SMGO-$\Delta$ routine.
%    \item Symbols with apex$^*$ indicate optimal values.
%    \item For the sake of brevity, symbols $g_s$ and $c_s$ synthetically indicate \textit{all} the constraint functions and values such that $g_s, s=1,\ldots,S$.
%\end{itemize}

\section{An overview on SMGO-$\Delta$}
\label{sec:background}
%In this section, the SMGO-$\Delta$ algorithm is summarized in its essential concepts, paramount for the self-consistency and understanding of the work and its \emph{meta} extension.
SMGO-$\Delta$ is an iterative optimization procedure specifically devised to tackle the following class of optimization problems
\begin{subequations}\label{eq:generic_opt}
\begin{align}
&\min_{X \in \mathcal{X}} ~~f(X)\\
&\quad \mbox{s.t. }~g_{s}(X)\geq 0,~~s=1,\ldots,S,
\end{align}
\end{subequations}
where the function $f(X):\mathcal{X} \mapsto \mathbb{R}$ one aims at minimizing is assumed to be \emph{unknown}. Note that the minimum, namely
\begin{equation}\label{eq:minimum}
X^{\star}=\underset{X \in \mathcal X, g_{s}(X)\geq 0}{\arg\min} f(X),
\end{equation}
should satisfy a set of inequality constraints, characterized by $g_s(X):\mathcal{G}_s \mapsto \mathbb{R}$, for $s=1,\ldots,S$, that are also supposed to be \emph{unknown}. The method is conceived to solve \eqref{eq:generic_opt} under the following assumptions.
\begin{assumption}\label{assumption1}
    The functions $f(\cdot)$ and $g_{s}(\cdot)$ are assumed to be Lipschitz continuous, namely they satisfy the following inequality:
    \begin{equation}
    |h(X_{1})-h(X_{2})|\leq \gamma_{h}\norm{X_1 - X_2}_{2},~~\forall X_1,X_2 \in \mathcal{X},
    \end{equation}
    with $h(\cdot)$ being a placeholder to indicate either of the two functions, and $\gamma_{h}>0$ being the associated Lipschitz constant.
\end{assumption}
\begin{assumption}\label{assumption2}
The feasible set $\mathcal{X}$ satisfies the following:
\begin{equation}
\mathcal{X} \cap \left\{ \bigcup_{s=1}^S \mathcal{G}_s \right\} \neq \emptyset,
\end{equation}
implying that the solution space $\mathcal{X}$ is not disjoint from the space of feasible solutions and, thus, that a solution to \eqref{eq:generic_opt} exists. 
\end{assumption}
By relying on these assumptions, SMGO-$\Delta$ addresses the problem in \eqref{eq:generic_opt} by iteratively performing three steps, summarized in the following, until a maximum number of iterations $n_{\mathrm{max}}$ is attained.

Let $\mathcal{D}^{(n)}$ be the set comprising all the information on the points $X$ explored up to the $n-1$-th iteration, \emph{i.e.,} the points themselves and the corresponding values of $f(\cdot)$ and $g_{s}(\cdot)$
\begin{align}
    & z^{(n)}=f\left(X^{(n)}\right),\\ 
    & c_{s}^{(n)}=g_{s}\left(X^{(n)}\right).
\end{align}
The information carried in $\mathcal{D}^{(n)}$ are initially used in SMGO-$\Delta$ to estimate the Lipschitz constant of the objective function as 
\begin{equation}
    \gamma_{f}^{(n)}=\max_{\substack{(X^{(i)},z^{(i)}) \in \mathcal{D}^{(n)},\\
    (X^{(j)},z^{(j)}) \in \mathcal{D}^{(n)}}} \left(\frac{|z^{(i)}-z^{(j)}|}{\|X^{(i)}-X^{(j)}\|},\underline{\gamma}_{f}\right),
\end{equation}
and that of each function characterizing the constraints as
\begin{equation}
   \gamma_{g_s}^{(n)}=\max_{\substack{(X^{(i)},c_s^{(i)}) \in \mathcal{D}^{(n)},\\
    (X^{(j)},c_s^{(j)}) \in \mathcal{D}^{(n)}}} \left(\frac{|c_s^{(i)}-c_s^{(j)}|}{\|X^{(i)}-X^{(j)}\|},\underline{\gamma}_{g_s}\right), 
\end{equation}
for $s=1,\ldots,S$. Note that both computation rely on a lower-bound on the Lipschitz constant (\emph{i.e.,} $\underline{\gamma}_{f}$ and $\{\underline{\gamma}_{g_{s}}\}_{s=1}^{S}$), initialized by the user and then iteratively replaced with the updated estimate of these constants over exploration\footnote{At the first iteration of SMGO-$\Delta$, the estimates of the Lipschitz constants are set at the user-defined lower-bound.}. These estimates are used to update the \emph{bounding functions}
\begin{subequations}\label{eq:bounding_functions}
\begin{align}
    &\overline{f}^{(n)}(X) = \min_{k=1,\ldots,n} \big( 
    z^{(k)} + \gamma^{(n)}_f \norm{X - X^{(k)}}
    \big),\\
    &\underline{f}^{(n)}(X) = \max_{k=1,\ldots,n} \big( 
    z^{(k)} - {\gamma}^{(n)}_f \norm{X - X^{(k)}}
    \big),\\
    &\overline{g}_{s}^{(n)}(X) = \min_{k=1,\ldots,n} \big( 
    c_{s}^{(k)} + {\gamma}^{(n)}_{g_{s}} \norm{X - X^{(k)}}
    \big),\\
    &\underline{g}_{s}^{(n)}(X) = \max_{k=1,\ldots,n} \big( 
    c_{s}^{(k)} - \Tilde{\gamma}^{(n)}_{g_s} \norm{X \!-\! X^{(k)}}
    \big),
\end{align}
\end{subequations}
for each $s \in\{1,\ldots,S\}$, the \emph{central approximations}
\begin{subequations}\label{eq:central_approx}
    \begin{align}
        &     \Tilde{f}^{(n)}(X) = \frac{1}{2} \Big(
    \overline{f}^{(n)}(X) + \underline{f}^{(n)}(X)
    \Big),\\
        &     \Tilde{g}_{s}^{(n)}(X)= \frac{1}{2} \Big(
    \overline{g_{s}}^{(n)}(X) + \underline{{g}_{s}}^{(n)}(X)
    \Big),~s\!=\!1,\ldots,S.
    \end{align}
\end{subequations}
and the \emph{uncertainty bounds}
\begin{subequations}
    \begin{align}
        & \lambda^{(n)}_f(X) = \overline{f}^{(n)}(X) - \underline{f}^{(n)}(X),\\
        & \lambda^{(n)}_{g_s}(X) = \overline{g_s}^{(n)}(X) - \underline{g_s}^{(n)}(X),~~s=1,\ldots,S.
    \end{align}
\end{subequations}
The central approximation and the uncertainty bounds are then exploited by SMGO-$\Delta$ to choose the next candidate point $X_{\theta}^{(n)}$, which is obtained by solving the following optimization problem
\begin{subequations}\label{eq:exploitation}
\begin{align}
    &\min_{X \in E^{(n)} \cup \mathcal{T}^{(n)}}~~\tilde{f}^{(n)}(X)-\beta\lambda_{f}^{(n)}(X)\\
    & \qquad \mbox{s.t. }~ \Delta \tilde{g}_{s}^{(n)}(X)\!+\!(1\!-\!\Delta) \underline{g}_{s}^{(n)}(X)\geq 0,~\forall s,
\end{align}
where $\beta,\Delta >0$ are tunable parameters, and
the state of space is restricted to the intersection of a (cumulative) set of candidate points $E^{(n)}$ and a \emph{trust region} $\mathcal{T}^{(n)}$ centered 
on the current estimate $X^{\star(n)}$ of the minima \eqref{eq:minimum}, here defined as
\end{subequations}
\begin{equation}\label{eq:trust}
\mathcal{T}^{(n)}=\left\{ X \in \mathcal{X}: \left\|X-X^{\star(n)}\right\|_{2}\leq v^{(n)}\right\}.
\end{equation}
With this candidate, the \emph{expected improvement} attained with the bounding function $\underline{f}^{(n)}(\cdot)$ is assessed by checking the following condition:
\begin{equation}\label{eq:expected-improvement-test}
\underline{f}^{(n)}(X_{\theta}^{(n)})\leq z^{\star(n)}-\alpha \gamma_{f}^{(n)}
\end{equation}
where $z^{\star(n)}$ is the estimate of the function value at the current estimated minima $X^{\star(n)}$, and $\alpha>0$ is another tunable parameter of the approach. If the latter inequality is satisfied, then $X_{\theta}^{(n)}$ becomes the new exploration point,\emph{i.e.,}
\begin{equation}
    X^{(n+1)} \leftarrow X_{\theta}^{(n)},
\end{equation}
otherwise \emph{exploration} is promoted over more uncertain regions by defining a new candidate point $X_\psi^{(n)}$ as follows
\begin{equation}
X_\psi^{(n)}= \underset{X\in E^{(n)}}{\arg\max}~~ \xi_{\psi}^{(n)}(X),
\end{equation}
where $\xi_{\psi}^{(n)}: E^{(n)} \rightarrow \mathbb{R}$ is the so-called \emph{exploration merit function}. Since this specific step is not modified by our META extension of the approach, we refer the reader to \cite{sabug2022smgo} for additional details. 
\section{PROBLEM STATEMENT}\label{sec:problem}
Let us assume that we have already exploited SMGO-$\Delta$ to retrieve the global optimum of $M$ functions, that are \emph{similar} to the one we aim at optimizing. Our goal is to exploit this similarity to enhance the performance and speed up the convergence of SMGO-$\Delta$ in tackling this new optimization instance.

To formally state this problem, let us define the concept of \emph{similarity} considered in this work. To this end, we define a \emph{family} of functions as follows:
\begin{equation}\label{eq:family}
\mathcal{F}_{\varphi}:=\{f: \Phi \times \mathcal{X} \rightarrow \mathbb{R} \mbox{ s.t. } z=f(X,\varphi)\},
\end{equation}
which, thus, comprises functions sharing the same parameterization, but (eventually) being characterized by different parameters $\varphi \in \Phi \subseteq \mathbb{R}^{n_{\varphi}}$.

\begin{definition}\label{Def:similarity}
Let $f_{1}(X,\varphi_{1})$ and $f_{2}(X,\varphi_{2})$ be two functions belonging to the family $\mathcal{F}(\varphi)$, and satisfying Assumption~\ref{assumption1}. These functions are said to be:
\begin{itemize}
    \item $\rho$-similar, if there exists $\rho<\infty$ such that the radius $\underline{\rho}$ of the \emph{smallest enclosing circle} of their minima $X_{1}^{\star}$ and $X_{2}^{\star}$ satisfies $\underline{\rho}\leq \rho$.
    \item $\zeta$-similar, if there exists $\zeta \in [0,\infty)$ such that $\overline{\zeta} \leq \zeta$, with $\overline{\zeta}$ being the maximum distance between the Lipschitz constants of the two functions, \emph{i.e.,}
    \begin{equation}
        \overline{\zeta}=\max |\gamma_{1}-\gamma_{2}|.
    \end{equation}
\end{itemize}
\end{definition}

Suppose now that $M$ constrained problems in the form of \eqref{eq:generic_opt} are solved with SMGO-$\Delta$ (considering the same number of optimization steps $n_{max}$), for different instances of the cost function (yet belonging to a family $\mathcal{F}_\varphi$ and being $\rho$-similar and $\zeta$-similar according to Definition~\ref{Def:similarity}) and the same set of constraints. As outcomes of these optimization routines, one can extract the resulting estimated minima $X_{i}^{\star}$ and minimum value $z_{i}^{\star}$, the set of explored states $\{X_{i}^{(n)}\}_{n=1}^{n_{max}}$, function values $\{z_{i}^{(n)}\}_{n=1}^{n_{max}}$ and constraints $\{c_{s,i}^{(n)}\}_{n=1}^{n_{max}}$, for $s=1,\ldots,S$, and the estimates of the Lipschitz contants
\begin{equation}
    \hat{\gamma}_{i}=\begin{bmatrix}
    \gamma_{f,i}^{(n_{max})} & \gamma_{g_{1},i}^{(n_{max})} & \cdots & \gamma_{g_{S},i}^{(n_{max})}
    \end{bmatrix}^{\top},
\end{equation}
with $i=1,\ldots,M$. These elements can be used to construct a \emph{META-dataset} $\mathcal{D}^{\mathrm{meta}}$, that, in turn, can be exploited when a new optimization problem \eqref{eq:generic_opt} with cost function $f(X,\varphi) \in \mathcal{F}_{\varphi}$, being $\rho$-similar and $\zeta$-similar to the $M$ functions already optimized and the same constraints characterizing the $M$ problems already tackled. Our goal can be formalized as follows.

Consider an instance of \eqref{eq:generic_opt}, with cost function belonging to a family of functions $\mathcal{F}_{\varphi}$. Suppose that such a function is $\rho$-similar and $\zeta$-similar to the objective functions of $M$ problems of the class \eqref{eq:generic_opt} that have already been solved. Under the assumption that the new problem shares the same constraints of all the others, we aim at exploiting the meta-dataset $\mathcal{D}^{\mathrm{meta}}$ to $(i)$ reduce the number of iterations required by SMGO-$\Delta$ to find the global optimum for the new instance of the problem, and $(ii)$ reducing the number of constraints violations throughout the new optimization.   

In this work, we translate this into exploiting $\mathcal{D}^{\mathrm{meta}}$ to initialize both the first evaluation point $X^{(1)}$ of the new instance of SMGO-$\Delta$ and the initial lower bounds on the Lipschitz constant $\underline{\gamma}_{f}$
required at the first iteration to compute the bounding functions and the central approximations in \eqref{eq:bounding_functions}-\eqref{eq:central_approx}.

\section{META-LEARNING FOR SMGO-$\Delta$}\label{sec:meta}
The performance SMGO-$\Delta$ are, by construction, shaped by the initial choices that the user has to perform. Here we focus on two crucial hyper-parameters, namely the initial exploration point and the lower bound for the cost's Lipschitz constant. Our idea is thus to extend this algorithm to its META version, relying on the intuition that information collected solving similar problems can help in improving the choices of these initial parameters, ultimately enhancing the optimization procedure. 

To this end, let us introduce the similarity vector
\begin{equation}\label{eq:similarity}
\mathcal{S}=\begin{bmatrix}
\mathcal{S}_{1} & \cdots & \mathcal{S}_{M}
\end{bmatrix}^{\top} \in \mathbb{R}^{M},
\end{equation}
whose elements satisfy the following relationships:
\begin{subequations}\label{eq:similarity1}
\begin{align}
& \mathcal{S}_{i} \geq 0,~~i=1,\ldots,M,\\
& \sum_{i=1}^{M} \mathcal{S}_{i}=1,
\end{align}
\end{subequations}
and that characterize the relative similarity between the problem we aim at solving and the $M$ ones whose features are included in the META-dataset $\mathcal{D}^{\mathrm{meta}}$. Note that, if the new problem has already been solved (and it corresponds to one associated to the $m$-th instance of the META-dataset), then we have $\mathcal{S}_{m}=1$ and $\mathcal{S}_{j}=0$, for all $j \in \{1,\ldots,M\}$, $j \neq m$. Let us then define the META-initialization of $X^{(1)}$ and $\underline{\gamma}_{f}$ as follows:
\begin{subequations}
\begin{align}
    & X^{(1),\mathrm{meta}}=\sum_{i=1}^{M}\mathcal{S}_{i}X_i^{\star(n_{max})},\label{eq:meta_init1}\\
    & \underline{\gamma}_{f}^{\mathrm{meta}}=\sum_{i=1}^{M}\mathcal{S}_{i}\gamma_{f,i}^{(n_{max})},\label{eq:meta_init2}
\end{align}
so that the initial exploration point and lower bound on the Lipschitz constant for META-SMGO-$\Delta$ are constructed as convex combinations of the estimates of the global minima and Lipschitz constants comprised in $\mathcal{D}^{\mathrm{meta}}$. 
\end{subequations}

We can now formalize the impact of these META initialization on the difference between the initial estimate of the minimal function value and the true minimum as follows.
\begin{proposition}\label{proposition1}
Consider problem \eqref{eq:generic_opt} and assume that its cost function $f(X)$ is $\rho$-similar and $\zeta$-similar (in the spirit of Definition~\ref{Def:similarity}) to a set of $M$ functions $\{f_{i}(X)\}_{i=1}^{M}$ for which \eqref{eq:generic_opt} has already been solved. Assume that $\rho \leq v^{(1)}$, with $v^{(1)}$ characterizing the trust region $\mathcal{T}^{(1)}$ according to \eqref{eq:trust}. Further assume that $X^{(1),\mathrm{meta}}$ in \eqref{eq:meta_init1} is a feasible initial exploration point. Under these assumption,for the first exploration point obtained with a META-initialization (in \eqref{eq:meta_init1}-\eqref{eq:meta_init2}), the following bound holds 
\begin{equation}\label{eq:prop_bound}
    z_{\theta}^{(1)}-z^{\star} \leq 2\rho(\gamma_{f}^{\mathrm{max}}+\zeta),
\end{equation}
where 
\begin{align}
z_{\theta}^{(1)}=\underline{f}^{(1)}(X_{\theta}^{(1)}),~~~\gamma_{f}^{\mathrm{max}}=\max_{i=1,\ldots,M}~~\gamma_{f,i}^{(n_{max})}.
\end{align}
\end{proposition}
\begin{proof}
Since the optimization subroutines used to populate the META-dataset are assumed to be $\rho$-similar, then the following holds
\begin{equation}\label{eq:bound_1}
\big\|X_{i}^{\star(n_{max})}-X^{\star}\big\|_{2} \leq 2\rho.
\end{equation}
The distance between $X^{(1),\mathrm{meta}}$ and the optimal solution can thus be bounded as
\begin{align}
\nonumber \big\|X^{(1)}-X^{\star}\big\|_{2}&=\big\|\sum_{i=1}^{M}\mathcal{S}_{i}X_{i}^{\star(n_{max})}-X^{\star}\big\|_{2}\\
\nonumber & \leq \sum_{i=1}^{M}\mathcal{S}_{i}\big\|X_{i}^{\star(n_{max})}-X^{\star}\big\|_{2}\\
& \leq 2\sum_{i=1}^{M}\mathcal{S}_{i}\rho=2\rho
\end{align}
where the second inequality holds thanks to the properties of the similarity vector $\mathcal{S}$ in \eqref{eq:similarity1} and the bound \eqref{eq:bound_1}. As $\rho\leq v^{(1)}$ by assumption, then $X^{*} \in \mathcal{T}^{(1)}$ and thus, \eqref{eq:expected-improvement-test} holds. This implies that
\begin{equation}
    z_{\theta}^{(1)}\leq z^{\star(1)}-\alpha\gamma_{f}^{(1)}
\end{equation}
Subtracting on both sides the actual value $z^{\star}$ of the function we aim at optimizing at the global optimimum we further obtain:
\begin{align}
   \nonumber z_{\theta}^{(1)}-z^{\star}& \leq z^{\star(1)}-\alpha\gamma_{f}^{(1)}-z^{\star}\\ 
   \nonumber & \leq  \gamma_{f}\|X^{(1)}-X^{\star}\|_{2}-\alpha\gamma_{f}^{(1)}\\
   &\leq 2\rho\gamma_{f}-\alpha\gamma_{f}^{(1)}\leq 2\rho\gamma_{f},
\end{align}
where $\gamma_{f}$ is the actual (unknown) Lipschitz constant of the function we are optimizing and the third inequality stems from the definition of Lipschitz continuity. Adding and subtracting on the right-hand-side of the previous inequality $2\rho\gamma_f^{(1)}$ we further obtain:
\begin{equation}
    z_{\theta}^{(1)}-z^{\star} \leq 2\rho\gamma_f^{(1)}+2\rho(\gamma_f^{(1)}-\gamma_{f}) \leq 2\rho(\gamma_{f}^{(1)}+\zeta).
\end{equation}
Since $\gamma_{f}^{(1)}=\underline{\gamma}^{\mathrm{meta}}_f$, based on the definition of $\underline{\gamma}^{\mathrm{meta}}_f$ in \eqref{eq:meta_init2}, it straightforwardly follows that $\gamma_{f}^{(1)}\leq \gamma^{\mathrm{max}}_{f}$, thus concluding the proof.
\end{proof}

The previous bound holds for any $\mathcal{S}$ satisfying \eqref{eq:similarity1}, yet it is of paramount importance for the similarity vector $\mathcal{S}$ to provide a reliable estimate of the similarity between the problem we aim at tracking and the $M$ ones that we have already solved to further reduce the distance of $z^{*(n)}_\theta$ from $z^*$. Toward this goal, in this paper we propose to iteratively evaluate similarity through the META-SMGO-$\Delta$ iterations as summarized in Algorithm~\ref{alg:cap}, thus considering an iteration varying similarity vector. 

Since at the beginning of the new optimization routine no information is available on the new function to be optimized, we initially impose
\begin{equation}
\mathcal{S}^{(1)}=\begin{bmatrix}
        \frac{1}{M} & \frac{1}{M} & \cdots & \frac{1}{M}
    \end{bmatrix}^{\top},
\end{equation}
not to (wrongly) prioritize any instance of the META-dataset with respect to the others. Since at each new iteration $n \in [1,n_{max}]$ of META-SMGO-$\Delta$ we have access to a new evaluated point $X^{(n)}$, they are incrementally employed to refine our initial guess. Specifically, we evaluate the unknown function we aim at optimizing at the current data point, namely  
\begin{equation*}
z^{(n)}=f\left(X^{(n)}\right). 
\end{equation*}
To extrapolate the updated similarity vector $\mathcal{S}^{(n)}$, the latter is then compared with the following natural neighbor interpolation\footnote{Accordingly, the more iterations $n_{max}$ are performed, the more $\hat{z}_{i}^{(n)}$ will be informative.} of 
\begin{align}\label{eq:interp}
    \hat{z}_i^{(n)} = \sum_{k=1}^{n_{max}} w_i(X_{i}^{(k)})z_{i}^{(k)},~~~~i=1,\ldots,M,
\end{align}
where we use a set of Laplacian weights \cite{bobach2009natural}. In particular, we solve the optimization problem (nested in the SMGO-$\Delta$ baseline routine):
\begin{subequations}\label{eq:tildeS_update}
    \begin{align}
    & \min_{\hat{\mathcal{S}}}~~\|z^{(n)}-\hat{\mathcal{S}}^{\top}\Hat{Z}\|_{2}^{2}\\
    & \quad \mbox{s.t. }~~ \hat{\mathcal{S}}_{i} \geq 0,~~i=1,\ldots,M,\\
    &\qquad \quad \sum_{i=1}^{M} \hat{\mathcal{S}}_{i}=1,
    \end{align}
\end{subequations}
where $\hat{Z}=\begin{bmatrix}\hat{z}_1^{(n)} & \cdots &\hat{z}_M^{(n)}\end{bmatrix}^{\top}$. Its solution $\hat{\mathcal{S}}^{\star}$ is then combined with the similarity vector available at the beginning of current iteration, namely
\begin{equation}\label{eq:s_theta}
    \mathcal{S}_{\theta}^{(n)}=\mathcal{S}^{(n)}+\tau^{n-1}\hat{\mathcal{S}}^{\star},
\end{equation}
where $\tau \in [0,1]$ is a \emph{discounting factor} introduced to promote smoothness in the similarity estimate over consecutive iterations. Clearly, \eqref{eq:s_theta} does not satisfy the properties of the similarity vector (see \eqref{eq:similarity1}). Accordingly, the elements of $\mathcal{S}_{\theta}^{(n)}\!\!=\!\!\mathcal{S}^{(n+1)}$ are then normalized, in order to lay in $[0,1]$. 

This updated estimate of similarity matrix is exploited to update $\underline{\gamma}_{f}$ as
\begin{equation*}
\underline{\gamma}_{f}=(\mathcal{S}_{\theta})^{\top}\begin{bmatrix} \underline{\gamma}_{f,1}^{(n_{max})}\\
\vdots\\ \underline{\gamma}_{f,M}^{(n_{max})},
\end{bmatrix}
\end{equation*}
\emph{before} the regular SMGO-$\Delta$ Lipschitz constant estimation. Additionally, after the exploitation subroutine of SMGO-$\Delta$, the estimated similarity vector is also used to promote sampling near the updated estimate of the optimal value as follows:
\begin{equation}\label{eq:exploitation-modification}
    X_\vartheta^{(n)} = (1 - \tau^{n-1})X_\theta^{(n)} + \tau^{n-1} (\mathcal{S}_{\theta}^{(n)})^{\top}\mathbf{X}^{\star(n_{max})}.
\end{equation}
where $\mathbf{X}^{\star(n_{max})}$ is a vector stacking the estimated global minima comprised in $\mathcal{D}^{\mathrm{meta}}$. This new point is then tested with the \emph{expected improvement} check \eqref{eq:expected-improvement-test}, and selected as the new sampling point $X^{(n+1)}$ if it passes this check. Otherwise, the exploration routine takes place with no difference with respect the original SMGO-$\Delta$. Note that, in this case, the discounting factor is of fundamental importance to exploit \emph{meta} information at the beginning of the optimization procedure, and gradually relying on the SMGO-$\Delta$ capabilities of finding the minimum when the number of iterations increases.
\begin{algorithm}
\caption{Meta SMGO-$\Delta$}\label{alg:cap}
\begin{algorithmic}
\Require $\mathcal{S}^{(1)},\mathcal{D}^{\mathrm{meta}}$
\State $X^{(1)}=\mathcal{S}^{\top}
\begin{bmatrix}X^{\star(n_{max})}_1 & \cdots & X^{\star(n_{max})}_M\end{bmatrix}^{\top}$
\State $\underline{\gamma}_f^{(1)}=\mathcal{S}^{\top} 
\begin{bmatrix}\gamma_{f,1}^{\star(n_{max})} & \cdots & \gamma_{f,M}^{\star(n_{max})}\end{bmatrix}^{\top}$
\While {$n \leq n_{\mathrm{max}}$}
    \State Evaluate $z^{(n)}=f(X^{(n)})$, $c_s^{(n)}=g_s(X^{(n)})$
    \For {$i = 1,\ldots, M$}
        \State Compute $\hat{z}_i^{(n)}$ as in \eqref{eq:interp}
    \EndFor
    \State Find $\hat{\mathcal{S}}^{\star}$ by solving \eqref{eq:tildeS_update}
    \State Update $\mathcal{S}_\theta^{(n)}=\mathcal{S}^{(n)}$ as in \eqref{eq:s_theta}
    \State Find $X^{(n+1)}$ with modified SMGO-$\Delta$ \eqref{eq:exploitation-modification}
\EndWhile
\end{algorithmic}
\end{algorithm}

\section{NUMERICAL EXPERIMENTS}\label{sec:experiments} 
In this preliminary work, advantages of meta-learning methodology are illustrated with the low-dimensional ($X \in \mathbb{R}^2$) example taken from \cite{sabug2022smgo} (in noiseless settings).
The class $\mathcal{F}_\varphi$ is the parameterized Styblinski-Tang function with offset, defined as
\begin{align}
\label{eq:f_class_example}
    f(X,\varphi_f) = \frac{1}{a_7}
    \Big(
    &a_1X_1^4 - a_2X_1^2+a_3X_1 +\\
    \nonumber&a_1X_2^4 - a_2X_2^2+a_3X_2 + a_8
    \Big).
\end{align}
The constraints are fixed and equal to
\begin{align}
    g_1(X) &= -4 + \norm{ X - [-2.90, 2.90]^{\top}}_2,\\
    g_2(X) &= \cos{\big(2 \norm{ X - [\ \ 2.90, 2.90]^{\top}}_2\big)}.
\end{align}
Each $a_j$, $j=1,\ldots,8$ is randomly obtained as $a_j = a_j^{\mathrm{o}}+\delta_j$, where $a_j^{\mathrm{o}}$ is the nominal value as reported in Tab. \ref{tab:nominal_values}, and $\delta_j$ is a perturbation $\delta_j \sim \mathcal{U}(-1,1)\cdot \delta_{\mathrm{max}}$.
The meta-data-set $\mathcal{D}^{\mathrm{meta}}$ (containing $M=10$ meta-functions) is generated optimizing each $f_m \in \mathcal{D}^{\mathrm{meta}}$ with SMGO-$\Delta$ (with parameters values reported in Tab. \ref{tab:smgo-params}). 
The generated functions are $\Tilde{\rho}=1.5$ $\Tilde{\zeta}=3000$ similar to each other.
In Fig. \ref{fig:nominal-optim}, the contour plot of nominal $f$ and the sampled points during the procedure are displayed, showing that the global minimum (green cross) is centered in $(-2.90,-2.90)$.
\begin{table}[ht]
    \caption{Nominal $\varphi_f$ values of the example}
    \label{tab:nominal_values}
    \centering
    \begin{tabular}{c|c|c|c|c|c|c|c}
        $a_1$ & $a_2$ & $a_3$ & $a_4$ & $a_5$ & $a_6$ & $a_7$ & $a_8$\\
        \hline
        1 & 16 & 5 & 1 & 16 & 5 & 2 & 80
    \end{tabular}
\end{table}
\begin{table}[ht]
    \caption{Parameters for $\mathcal{D}^{\mathrm{meta}}$ generation}
    \label{tab:smgo-params}
    \centering
    \begin{tabular}{c|c|c|c|c|c|c|c}
         n. iter & $\Delta$ & $\beta$ & $\alpha$ & $X^{(1)}$ & $\delta_{\mathrm{max}}$ & $\Tilde{\rho}$ & $\Tilde{\zeta}$\\
         \hline
         500 & 0.5 & 0.1 & 0.1 & (0.4775,0.0667) & 75\% & 1.5 & 3000
    \end{tabular}
\end{table}
\begin{figure}[ht]
    \centering
    \includegraphics{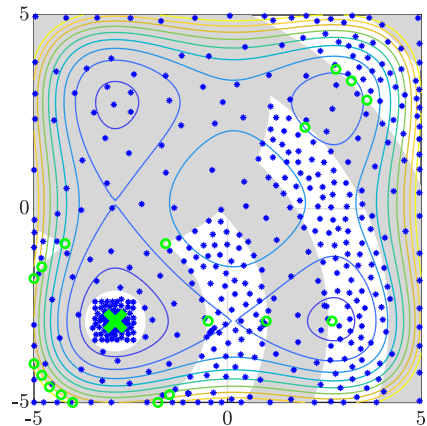}
    \caption{Contour plot of nominal $f$. In grey the unfeasible areas due to $g_1$ and $g_2$. The green cross indicate the global minimum while the green circles the local ones. Blue points the sampled $X$ in SMGO-$\Delta$ optimization.}
    \label{fig:nominal-optim}
\end{figure}
\subsection{Limit case: $f \equiv f_{\Tilde{m}}$}
The case where the new $f$ is equivalent to an already-optimized function $f_{\Tilde{m}} \in \mathcal{D}^{\mathrm{meta}}$ is first tested, to verify that the data-driven algorithm identifies the correct similarity for $\mathcal{S}$.
This holds true for a test with 10 repeated experiment, where the reference function changes such that $\Tilde{m} = 1,\ldots,10$. 
On average, the similarity coefficient associated to the correct $f_{\Tilde{m}}$ is 0.75, with a standard deviation of 0.25 as shown in Fig. \ref{fig:similarity} (as a consequence, the other coefficients of $\mathcal{S}$ are very small).\\
The non-perfect convergence to 1 is justified by the estimates of $\Hat{Z}$ that are obtained with interpolation. Nonetheless, for these limit cases convergence to the global optimum occurs almost immediately.
%
% As an example, Fig. \ref{fig:similarity} displays the trajectories of $S$ when $f \equiv f_9$ (the trend is analogous for different $\Tilde{m}$ reference functions). 
%
% tested on 10 experiments where the reference function $\Tilde{m}=1,\ldots,10$ changes. $S_{ref}$ is the vector where $S_{\Tilde{m}} = 1$ and the other elements are $0$.
%
\begin{figure}[ht]
    \centering
    \includegraphics{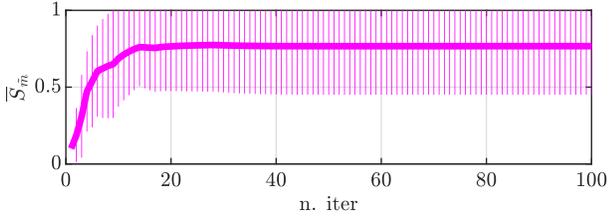}
    \caption{Trajectory of the average coefficient values $S_{\Tilde{m}}$, where $f = f_{\Tilde{m}}$}
    \label{fig:similarity}
\end{figure}

\subsection{General case}
\textit{SMGO-$\Delta$ with meta-learning} (META) is tested on $N=10$ experiments, where each new $f_n, n=1,\ldots,N$ to be optimized is generated according to \eqref{eq:f_class_example}.
% with $\delta_{\mathrm{meta}}=75\%$ (constraints $g_1$ and $g_2$ are unmodified).
Each $f_n$ is optimized with META ($\tau=0.9$) and compared to standard SMGO-$\Delta$ optimization, with $n_{\mathrm{max}}=100$ iterations.
META is initialized such that 
$X^{(1)}=(\mathcal{S}^{(1)})^{\top}\mathbf{X}^{\star(n_{max})}$, $\underline{\gamma}_f=(\mathcal{S}^{(1)})^{\top}\bm{\gamma}_{f}^{(n_{max})}$, with $\mathcal{S}=[0.1, \ldots, 0.1]^{\top}$. For fairness, in this preliminary work Lipschitz constants of the constraints $\underline{\gamma}_{g_1},\underline{\gamma}_{g_1} = 10^{-6}$ are initialized as if no meta-information is available -- though it is clear that a prior that can be exploited to improve their initialization exists. Future work will derive a rigorous meta-formulation that also considers the constraints.
Trajectories of the average \emph{best} $\overline{z}^{*(n)}$ over the $N$ experiments obtained with META and SMGO-$\Delta$ are shown in Fig. \ref{fig:meta_best_z}, from which we can appreciate that META significantly reduces the iterations required to reach the global minimum ($z^{(1)}_\theta - z^* = 15.10 \leq 2\cdot1.5\cdot (2817 + 3000)$, largely satisfying condition of Prop. \ref{proposition1}).
In addition, the average number of infeasible samples is significantly reduced (25\%, Fig. \ref{fig:boxplot-constraint}), even if no constraint prior is employed. 
Finally, notice how the initialization of $\underline{\gamma}_f$ is closer to the final estimate of $\gamma_f$ at convergence (Fig. \ref{fig:meta_gamma}), promoting a more targeted search of the regions where the minimum is expected from the start.
This holds also for functions where, due to the perturbation, the location of the (feasible) global minimum varies significantly from the nominal optimum (Fig. \ref{fig:contour-meta-example}).
Also, the update of $\mathcal{S}$, in case the META-initialization of $\underline{\gamma}^{(1)}$ is an overestimate of the true $\gamma$, can compensate the error such that $\Tilde{\gamma}^{(n)}\leq\Tilde{\gamma}^{(n+1)}$.
\begin{figure}[ht]
    \centering
    \includegraphics{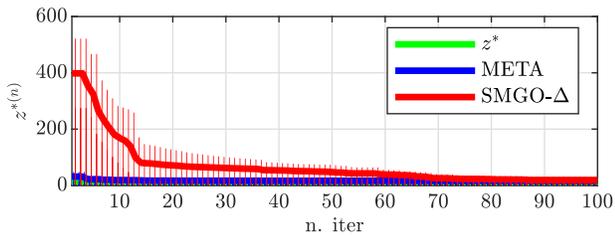}
    \caption{Average trajectory of $z^{*(n)}$ during optimization}
    \label{fig:meta_best_z}
\end{figure}
\begin{figure}[ht]
    \centering
    \includegraphics{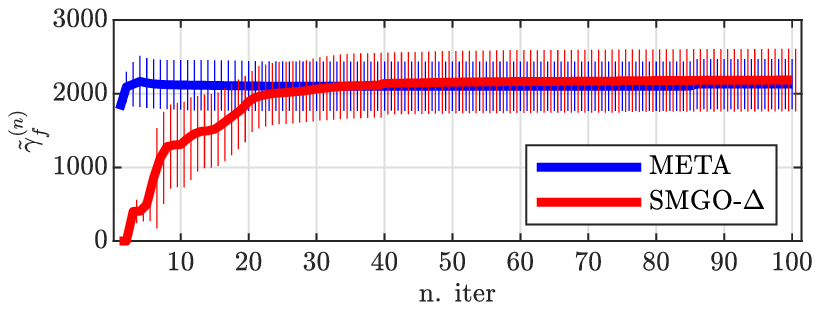}
    \caption{Average trajectory of $\Tilde{\gamma}^{(n)}_f$ during optimization}
    \label{fig:meta_gamma}
\end{figure}
\begin{figure}[ht]
    \centering
    \begin{subfigure}{.49\linewidth}
    \centering
    \includegraphics[scale=1]{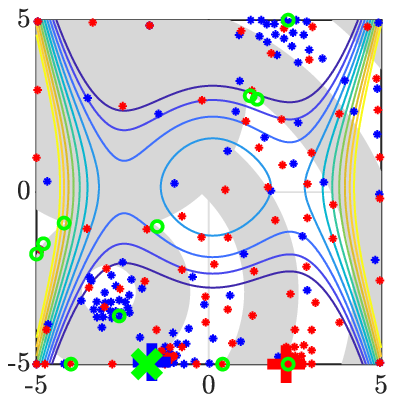}
    \caption{}
    \label{fig:contour-meta-example}
    \end{subfigure}%
    \begin{subfigure}{.49\linewidth}
    \centering
    \includegraphics[scale=1]{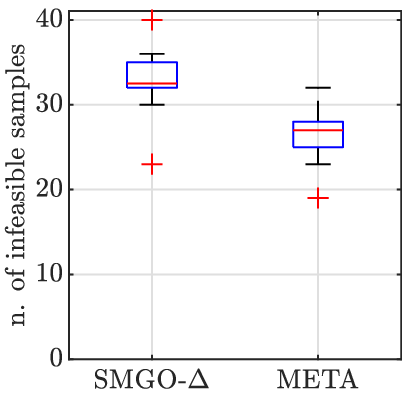}
    \caption{}
    \label{fig:boxplot-constraint}
    \end{subfigure}
    \caption{Contour plot of $f$, with sampled points with META (blue) and SMGO-$\Delta$ (red) (a) and average number of infeasible points during the optimization (b).}
\end{figure}

\subsection{Sensitivity to $M$}
Sensitivity to size $M$ of  the meta-data-set is tested for $M=5,10,20,40$, on the $N=10$ test experiments.. We found that, on average, both distance from the optimal value $z^{(1)} - z^*$ (Fig. \ref{fig:sensitivity_M_1}) and constraint violations (Fig. \ref{fig:sensitivity_M_2}) are reduced for greater $M$. Nonetheless, for $M=40$ there is a settling of these improvements. Numerical results confirms the intuition that a greater number of examples is informative up to the point where they become redundant.
\begin{figure}[H]
    \centering
    \begin{subfigure}{\linewidth}
    \centering
    \includegraphics[scale=1]{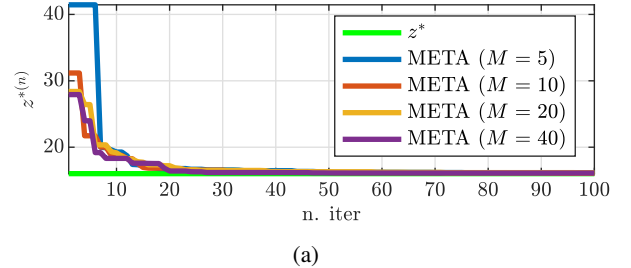}
    \caption{}
    \label{fig:sensitivity_M_1}
    \end{subfigure}%
    \\
    \begin{subfigure}{\linewidth}
    \centering
    \includegraphics[scale=1]{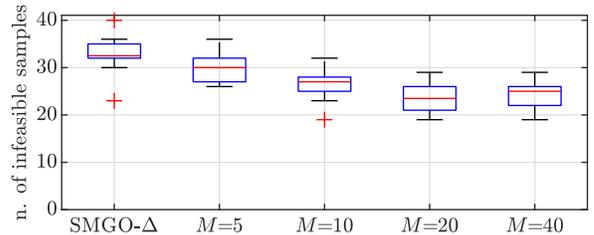}
    \caption{}
    \label{fig:sensitivity_M_2}
    \end{subfigure}
    \caption{Sensitivity to $M$ -- Average $z^{*(n)}$ over $N=10$ experiments (a) and average number of constraint violations (b) over $N=10$ experiments.}
\end{figure}

\subsection{Sensitivity to $\tau$}
Sensitivity to discounting factor $\tau$ is tested for $\tau=0.3,0.5,0.7,0.9,0.99$, on the $N=10$ test experiments. We recall that the purpose of $\tau$ is to trade-off exploitation of similarity and SMGO-$\Delta$ capabilities. The advantage of higher values of $\tau$, i.e., relying to similarity for more iterations, is appreciated for $n \leq 6$, where average closeness to $z^*$ is reduced (Fig. \ref{fig:sensitivity_tau_1}). However, average convergence is reached almost simultaneously in the experiments.
The number of constrain violations is also reduced for higher values of $\tau$ (Fig. \ref{fig:sensitivity_tau_2}), though for $\tau=0.99$, dispersion increases since the algorithm heavily relies on the similarity).
\begin{figure}[H]
    \centering
    \begin{subfigure}{\linewidth}
    \centering
    \includegraphics[scale=1]{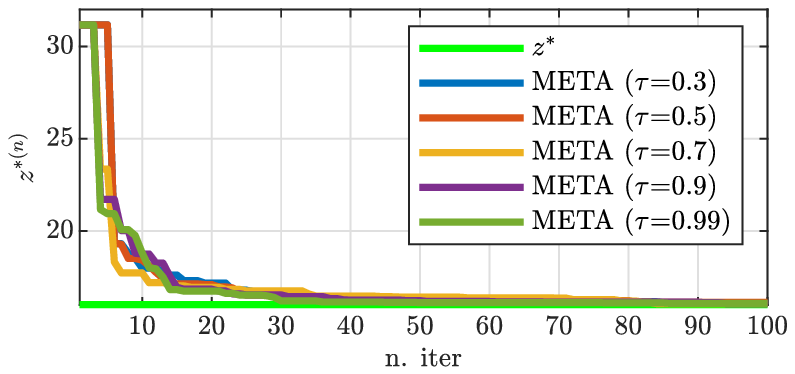}
    \caption{}
    \label{fig:sensitivity_tau_1}
    \end{subfigure}%
    \\
    \begin{subfigure}{\linewidth}
    \centering
    \includegraphics[scale=1]{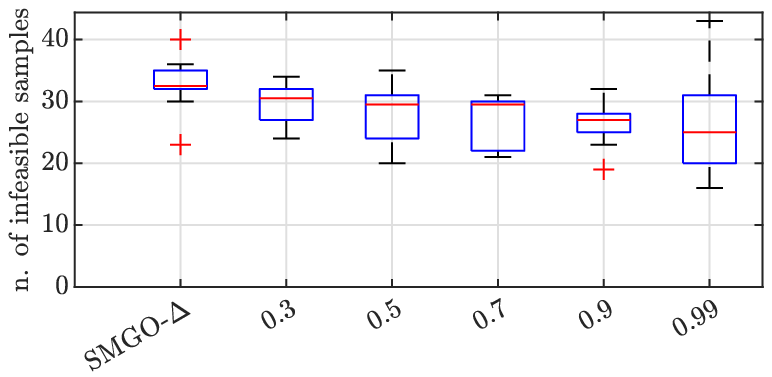}
    \caption{}
    \label{fig:sensitivity_tau_2}
    \end{subfigure}
    \caption{Sensitivity to $\tau$ -- Average $z^{*(n)}$ over $N=10$ experiments (a) and average number of constraint violations (b) over $N=10$ experiments.}
\end{figure}

\section{CONCLUSIONS}
The work proposes the application of a meta-learning rationale to SMGO-$\Delta$, with the objective of exploiting prior experience to make the optimization more efficient. The proposed method exploits a specific notion of similarity with past optimization problems as a prior to initialize two paramount hyper-parameters of the nominal method, namely $X^{(1)}$ and $\underline{\gamma}_f$.
We demonstrate that such an initialization results in a theoretical bound on the closeness of $z^{(1)}_\theta$ to the global minimum $z^*$. Numerical experiments confirm the theoretical findings and demonstrate the superiority with respect to the standard SMGO-$\Delta$, in terms of speed of convergence and reduction of the number of constraint violations during the optimization routine.

Further works will focus on the application of meta-learning approaches to other optimization techniques and on the extension of the proposed approach to non-fixed constraints.

\addtolength{\textheight}{-12cm}   % This command serves to balance the column lengths
                                  % on the last page of the document manually. It shortens
                                  % the textheight of the last page by a suitable amount.
                                  % This command does not take effect until the next page
                                  % so it should come on the page before the last. Make
                                  % sure that you do not shorten the textheight too much.

%%%%%%%%%%%%%%%%%%%%%%%%%%%%%%%%%%%%%%%%%%%%%%%%%%%%%%%%%%%%%%%%%%%%%%%%%%%%%%%%

%%%%%%%%%%%%%%%%%%%%%%%%%%%%%%%%%%%%%%%%%%%%%%%%%%%%%%%%%%%%%%%%%%%%%%%%%%%%%%%%

%%%%%%%%%%%%%%%%%%%%%%%%%%%%%%%%%%%%%%%%%%%%%%%%%%%%%%%%%%%%%%%%%%%%%%%%%%%%%%%%
%\section*{APPENDIX}

%\section*{ACKNOWLEDGMENT}

%%%%%%%%%%%%%%%%%%%%%%%%%%%%%%%%%%%%%%%%%%%%%%%%%%%%%%%%%%%%%%%%%%%%%%%%%%%%%%%%

\bibliographystyle{IEEEtran}
\bibliography{bib}

\end{document}